\documentclass[11pt,leqno]{amsart}

\usepackage{amsmath}
\usepackage{amssymb}
\usepackage{amsthm}
\usepackage{amsfonts}
\usepackage{mathtools}
\usepackage{ascmac}
\usepackage{spalign}
\usepackage{fancybox} 
\usepackage[mathscr]{eucal}
\usepackage{footnpag}
\usepackage{siunitx}
\usepackage{multicol}
\usepackage{bbm}
\usepackage{framed}
\usepackage{here}

\usepackage{graphicx}
\usepackage{xcolor}
\usepackage{wrapfig}
\usepackage{float}

\usepackage{url}
\usepackage{comment}

\theoremstyle{definition}
\newtheorem{theorem}{Theorem}[section]

\newtheorem{proposition}[theorem]{Proposition}
\newtheorem{lemma}[theorem]{Lemma}
\newtheorem{corollary}[theorem]{Corollary}
\newtheorem{remark}[theorem]{Remark}

\newfont{\bg}{cmr9 scaled\magstep4}
\newcommand{\bigzero}{
\smash{\lower1.0ex\hbox{\bg 0}}}

\newcommand{\ctext}[1]{\raise0.2ex\hbox{\textcircled{\scriptsize{#1}}}}

\title[Average Hitting Times and Recurrence Structures I]
{Average Hitting Times and Recurrence STRUCTURES I: Powers of Cycle Graphs}

\author[Miezaki]{Tsuyoshi Miezaki}
\address{		Faculty of Science and Engineering, 
		Waseda University, 
		Tokyo 169--8555, Japan
}
\email{miezaki@waseda.jp} 

\author[Tamura]{Shunya Tamura*}
\thanks{*Corresponding author}
\address{	Okegawa City Okegawa West Junior High School, Saitama, 363-0027, Japan
}
\email{shunya.tamura059@gmail.com}



\begin{document}
\begin{abstract}
We investigate the average hitting times of simple random walks on the $k$-th power graph $C_N^k$ of the cycle graph $C_N$.

First, we show that the average hitting times are characterized by a difference equation corresponding to the graph Laplacian. Next, by using the cyclic symmetry of $C_N^k$, we derive a spectral representation via Fourier analysis. Furthermore, by applying factorization and partial fraction decomposition of the corresponding difference operator, we obtain an explicit formula for the average hitting times consisting of a quadratic term and finitely many correction terms.

These correction terms are described by second-order linear recurrence sequences associated with the characteristic polynomials, and can be regarded as natural generalizations of Fibonacci-type sequences. As a consequence, our formulas recover the known results for cycle graphs and squares of cycle graphs in a unified way.
Moreover, from the formulas obtained for average hitting times, we derive explicit formulas for the effective resistances, the numbers of spanning trees, the numbers of two-component spanning forests, and the numbers of spanning trees of vertex-identified graphs. In particular, for the third power graph $C_N^3$ of the cycle graph, all of these quantities are written explicitly in terms of complex conjugate Fibonacci-type sequences.

Our results clarify structural relations between random walk quantities and combinatorial quantities on cycle power graphs.
\end{abstract}
\maketitle

\section{Introduction} \label{sec01}

The relation between random walks and electric networks is one of the fundamental topics in combinatorics, probability theory, and spectral graph theory. In particular, it is well known that there exist deep relations among average hitting times, effective resistances, the numbers of spanning trees, and the numbers of two-component spanning forests
\cite{ChandraEtAl1989,ChebotarevShamis1997,Kirchhoff1847,NashWilliams1959,Wu2004}.

For the simple random walk on the cycle graph $C_N$, it has long been known that the average hitting time is given by
\[
h(0,\ell)=\ell(N-\ell).
\]
On the other hand, for the square graph $C_N^2$ of the cycle graph, it is known that the average hitting times can be expressed in terms of Fibonacci numbers.
N.~Chair \cite{Chair2014} derived explicit formulas for the average hitting times by using the method of effective resistances.
Furthermore,
Doi et al.~\cite{DoiEtAl2022}
gave a simpler expression in terms of Fibonacci numbers and clarified relations between average hitting times and Fibonacci structures: 
\[
h(0,\ell)
=
\frac{2}{5}\,\ell(N-\ell)
+
\frac{4}{5}\,N
\frac{F_{2\ell}F_{2(N-\ell)}}{F_{2N}}.
\]

These results naturally raise the following question: whether analogous Fibonacci-type structures appear for the graph \(C_N^3\) and, more generally, for the \(k\)-th power graph \(C_N^k\) of the cycle graph. 
However, such generalizations do not seem to follow easily from the method of \cite{DoiEtAl2022}.

In general, by using the spectral decomposition of the Laplacian matrix, the average hitting time on $C_N^k$ can be written as
\[
h(0,\ell)
=
\sum_{j=1}^{N-1}
\frac{
1-\cos(2\pi j \ell/N)
}{
1-\frac1k\sum_{r=1}^k\cos(2\pi jr/N)
}.
\]
However, this expression remains as a finite sum, and the underlying structure is not necessarily clear.

The purpose of this paper is to investigate the average hitting times on the cycle power graph $C_N^k$ 
from the viewpoint of difference operators and Fourier analysis, and to derive explicit and structural formulas.
In particular, by factorizing the corresponding difference operator, we show 
%
that the average hitting time on $C_N^k$ is expressed as
\[
h(0,\ell)
=
\frac{6}{(k+1)(2k+1)}\,\ell(N-\ell)
+
N\sum_{a=1}^{k-1}
A_a
\frac{
V_\ell^{(a)}V_{N-\ell}^{(a)}
}{
V_N^{(a)}
},
\]
where each sequence $V_n^{(a)}$ satisfies the second-order linear recurrence relation
\[
V_{n+1}^{(a)}
=
\gamma_aV_n^{(a)}-V_{n-1}^{(a)}.
\]
These sequences can be regarded as natural generalizations of Fibonacci-type sequences.

As consequences of our results, we recover the classical formula
\[
h(0,\ell)=\ell(N-\ell)
\]
for the case $k=1$, and derive the Fibonacci-type formulas in
\cite{Chair2014,DoiEtAl2022}
for the case $k=2$ in a unified way.
Furthermore, in the case $k=3$, complex conjugate second-order recurrence sequences naturally appear.
Using them, we obtain explicit formulas for the average hitting times, effective resistances, the numbers of spanning trees, the numbers of two-component spanning forests, and the numbers of spanning trees of vertex-identified graphs.
To the best of our knowledge, such explicit formulas for these quantities on $C_N^3$ have not appeared previously.

On the other hand, it is still unclear why such second-order linear recurrence relations naturally appear for cycle power graphs, and how these structures are related to Fibonacci-type phenomena.
A conceptual explanation will be discussed in forthcoming work. 

The rest of this paper is organized as follows.
In Section \ref{sec02}, we prepare basic facts on cycle power graphs, random walks, and average hitting times.
In Section 2.2, we study the factorization of the corresponding difference operators and the associated second-order recurrence structures.
In Section \ref{sec04}, we derive spectral representations of average hitting times by Fourier analysis.
Section \ref{sec05} gives explicit formulas for average hitting times by using Fourier analysis and periodization.
In Section \ref{sec06}, we derive formulas for effective resistances, spanning trees, two-component spanning forests, and spanning trees of vertex-identified graphs.
Finally, in Section \ref{sec07}, we discuss the cases of $C_N$, $C_N^2$, and $C_N^3$ in detail.
In particular, for $C_N^3$, we give completely explicit formulas for all of these quantities.

\section{Preliminaries} \label{sec02}

\subsection{Basic definitions}

In this section, we prepare notation and basic definitions used throughout this paper.

Let
\[
\mathbb{Z}_N=\{0,1,\dots,N-1\}
\]
be the cyclic group modulo $N$.
All computations are considered on $\mathbb{Z}_N$.

For a positive integer $k$, the $k$-th power graph $C_N^k$ of the cycle graph is defined by the vertex set
\[
V(C_N^k)=\mathbb{Z}_N
\]
and the edge set
\[
E(C_N^k)
=
\left\{
\{i,j\}
\mid
i-j\equiv \pm r \pmod N,
\ 1\le r\le k
\right\}.
\]
In other words, in $C_N^k$, each vertex is adjacent to all vertices whose cycle distance is
\[
1,2,\dots,k.
\]
Throughout this paper, we mainly consider the case
\[
N\ge 2k+1.
\]
Under this assumption, $C_N^k$ is $2k$-regular.

Next, we consider the simple random walk on $C_N^k$.
If the walker is at a vertex $x$ at time $t$, then at the next step the walker moves to one of
\[
x\pm1,\ x\pm2,\dots,\ x\pm k
\]
with equal probability $1/(2k)$.
We denote this stochastic process by
\[
(X_t)_{t\ge0}.
\]
For a vertex $q$, the hitting time of $q$ is defined by
\[
T_q
=
\inf\{t\ge0\mid X_t=q\}.
\]
The average hitting time from a vertex $p$ to a vertex $q$ is defined by
\[
h(p,q)
=
\mathbb{E}_p[T_q].
\]
In particular, we mainly treat the average hitting time from the vertex $0$ to the vertex $\ell$: $h(0,\ell)$.

Finally, we define the shift operator $T$ by
\[
(Tf)(i)=f(i+1),
\]
where all indices are considered modulo $N$.
This operator will be used later to describe difference operators and spectral representations.

\subsection{Basic Equations and Difference Operators}
\label{sec03}

In this section, we derive the basic equation satisfied by the average hitting times on $C_N^k$ and investigate its structure as a difference operator.


\begin{proposition} \label{prop:basic-equation}
For the simple random walk on $C_N^k$, let
\[
h(i)=h(0,i)
\]
be the average hitting time from the vertex $0$ to the vertex $i$.
Then $h$ satisfies
\[
h(0)=0,
\]
and
\[
2k\,h(i)-\sum_{r=1}^k \bigl(h(i-r)+h(i+r)\bigr)=2k
\qquad (i\neq 0).
\]
Equivalently,
\[
L_k h = 2k\,\mathbf{1},
\]
where
\[
L_k
=
2kI-\sum_{r=1}^k(T^r+T^{-r}),
\]
and $\mathbf{1}$ denotes the constant function.
\end{proposition}

\begin{proof}
By the Markov property, for $i\neq0$, we have
\[
h(i)
=
1+\frac{1}{2k}
\sum_{r=1}^k
\bigl(h(i-r)+h(i+r)\bigr).
\]
Multiplying both sides by $2k$ and rearranging the terms, we obtain
\[
2k\,h(i)
-
\sum_{r=1}^k
\bigl(h(i-r)+h(i+r)\bigr)
=
2k.
\]
Since the left-hand side coincides with $(L_kh)(i)$, it follows that
\[
(L_kh)(i)=2k.
\]
Moreover, by definition, we have
\[
h(0)=0.
\]
\end{proof}


Let 
\[
\widetilde{\Phi}_k(z)
=
2k-\sum_{r=1}^k(z^r+z^{-r}).
\]
Then we have 
\[
L_k=\widetilde{\Phi}_k(T),
\]

Next, in order to investigate the structure of $\widetilde{\Phi}_k$, we put
\[
x=z+z^{-1}.
\]
By the Chebyshev recursion, each term $z^r+z^{-r}$ can be expressed as a polynomial in $x$. Thus we define
\[
P_r(x):=z^r+z^{-r}.
\]
Hence we have
\[
\widetilde{\Phi}_k(z)
=
2k-\sum_{r=1}^k(z^r+z^{-r})
=
2k-\sum_{r=1}^k P_r(x),
\]
and therefore there exists a polynomial $\Phi_k(x)$ satisfying
\[
\widetilde{\Phi}_k(z)=\Phi_k(x).
\]


\begin{lemma} \label{lem32}
The polynomial $\Phi_k(x)$ can be factorized as
\[
\Phi_k(x)=(2-x)\Psi_k(x),
\]
where $\Psi_k(x)$ is a polynomial of degree $k-1$.
\end{lemma}

\begin{proof}
Putting
\[
x=2\cos\theta,
\]
we have
\[
z^r+z^{-r}=2\cos(r\theta).
\]
Hence
\[
\Phi_k(x)
=
2k-2\sum_{r=1}^k\cos(r\theta).
\]

Using the Chebyshev polynomial $T_r$ of the first kind,
\[
\cos(r\theta)
=
T_r(\cos\theta)
=
T_r(x/2),
\]
and therefore
\[
\Phi_k(x)
=
2k-2\sum_{r=1}^k T_r(x/2).
\]
Since $T_r(x/2)$ is a polynomial of degree $r$, the polynomial $\Phi_k(x)$ has degree $k$.

Next, substituting $x=2$, which corresponds to $\theta=0$, we obtain
\[
\Phi_k(2)
=
2k-2\sum_{r=1}^k\cos(0)
=
2k-2k
=
0.
\]
Hence, by the factor theorem, $(x-2)$ is a factor of $\Phi_k(x)$.

Therefore we can write
\[
\Phi_k(x)
=
(2-x){\Psi}_k(x).
\]
\end{proof}

\subsection{Spectral Representation} \label{sec04}

In this section, we give a spectral representation of the average hitting times on $C_N^k$.

\begin{lemma} \label{lem41}
We identify the vertex set of $C_N^k$ with $\mathbb{Z}_N$.
Let
\[
\phi_j(x)=\exp\!\left(\frac{2\pi i jx}{N}\right),
\qquad j=0,1,\dots,N-1
\]
be the Fourier basis.
Then the shift operator $T$ satisfies
\[
T\phi_j = e^{2\pi i j/N}\phi_j.
\]
Hence the operator
\[
L_k = 2kI - \sum_{r=1}^k (T^r + T^{-r})
\]
is diagonalized by the Fourier basis, and its eigenvalues are
\[
\lambda_j
=
2k - 2\sum_{r=1}^k \cos\!\left(\frac{2\pi jr}{N}\right),
\qquad j=0,1,\dots,N-1.
\]
In particular, $\lambda_0=0$, and $\lambda_j>0$ for $j\neq0$.
\end{lemma}

\begin{proof}
For the Fourier basis, we have
\[
T\phi_j(x)=\phi_j(x+1)=e^{2\pi i j/N}\phi_j(x).
\]
Hence
\[
T^r\phi_j = e^{2\pi i jr/N}\phi_j,
\qquad
T^{-r}\phi_j = e^{-2\pi i jr/N}\phi_j.
\]
Therefore,
\[
L_k\phi_j
=
\left\{
2k-\sum_{r=1}^k
\left(e^{2\pi i jr/N}+e^{-2\pi i jr/N}\right)
\right\}\phi_j.
\]
By Euler's formula, the desired eigenvalue formula follows.
\end{proof}

The following spectral representation is standard; see, for example, [1, Chapter 2, Lemma 12] and [13, Proposition 2.4]. We include a proof adapted to the present setting for completeness.
\begin{proposition} \label{prop41}
For the simple random walk on $C_N^k$, the average hitting time from the vertex $0$ to the vertex $\ell$ is given by
\[
h(0,\ell)
=
\sum_{j=1}^{N-1}
\frac{1-\cos\left(\frac{2\pi j\ell}{N}\right)}
{1-\frac{1}{k}\sum_{r=1}^k\cos\left(\frac{2\pi jr}{N}\right)}.
\]
\end{proposition}

\begin{proof}
The following arguments appeared in 
{\cite[Chapter 2, Lemma 12]{AF2002}, \cite[Proposition 2.4 ]{Saloff-Coste-Wang2025}}. 
We identify the vertex set of $C_N^k$ with $\mathbb{Z}_N$.
Let
\[
\phi_j(x)=\exp\left(\frac{2\pi i jx}{N}\right),
\qquad j=0,1,\dots,N-1
\]
be the Fourier basis.
By Lemma \ref{lem41}, the eigenvalues of $L_k$ are
\[
\lambda_j
=
2k-2\sum_{r=1}^k\cos\left(\frac{2\pi jr}{N}\right),
\]
and
\[
L_k\phi_j=\lambda_j\phi_j.
\]

For a vertex $y$, let $\delta_y$ be the delta function defined by
\[
\delta_y(x)
=
\begin{cases}
1 & (x=y),\\
0 & (x\neq y).
\end{cases}
\]
By Fourier inversion, 
\[
\delta_y(x)
=
\frac{1}{N}\sum_{j=0}^{N-1}\phi_j(x-y), 
\]
and therefore
\[
\delta_y(x)-\frac1N
=
\frac1N\sum_{j=1}^{N-1}\phi_j(x-y).
\]

Now define the Green function of $L_k$ by
\[
G(x,y)
=
\frac1N\sum_{j=1}^{N-1}\frac{\phi_j(x-y)}{\lambda_j}.
\]
Applying $L_k$ to this definition, we obtain
\[
L_kG(x,y)
=
\frac1N\sum_{j=1}^{N-1}
\frac{L_k\phi_j(x-y)}{\lambda_j}.
\]
Since $L_k\phi_j=\lambda_j\phi_j$, we have
\[
L_kG(x,y)
=
\frac1N\sum_{j=1}^{N-1}\phi_j(x-y).
\]
Therefore, by the Fourier expansion of the delta function,
\[
L_kG(x,y)=
\delta_y(x)-\frac1N.
\]

Let
\[
h_y(x)=h(x,y)
\]
be the average hitting time from $x$ to $y$.
By the basic equation of the average hitting time,
\[
h_y(y)=0
\]
and
\[
(L_kh_y)(x)=2k
\qquad (x\neq y)
\]
hold.
We define
\[
\widetilde{h}_y(x)
=
2kN\{G(y,y)-G(x,y)\}.
\]
Then clearly
\[
\widetilde{h}_y(y)=0.
\]
Moreover, since $G(y,y)$ is constant as a function of $x$, we have
\[
L_kG(y,y)=0.
\]
Thus
\[
(L_k\widetilde{h}_y)(x)
=
-2kN(L_kG(x,y)).
\]
Using
\[
L_kG(x,y)=\delta_y(x)-\frac1N,
\]
we see that, for $x\neq y$,
\[
(L_k\widetilde{h}_y)(x)
=
-2kN\left(-\frac1N\right)
=
2k.
\]
Hence $\widetilde{h}_y$ satisfies
\[
\widetilde{h}_y(y)=0,
\qquad
(L_k\widetilde{h}_y)(x)=2k
\quad (x\neq y).
\]
This is the same boundary value problem which characterizes the average hitting time $h_y$.
Since the kernel of $L_k$ on periodic functions consists only of constants and both functions vanish at $y$, the solution is unique. Therefore,
\[
h_y(x)
=
\widetilde{h}_y(x)
=
2kN\{G(y,y)-G(x,y)\}.
\]

Putting $x=0$ and $y=\ell$, we have
\[
h(0,\ell)
=
2kN\{G(\ell,\ell)-G(0,\ell)\}.
\]
By translation invariance,
\[
G(\ell,\ell)=G(0,0).
\]
Therefore
\[
h(0,\ell)
=
2kN\{G(0,0)-G(0,\ell)\}.
\]

By the definition of the Green function,
\[
G(0,0)
=
\frac1N
\sum_{j=1}^{N-1}
\frac{1}{\lambda_j},
\]
and
\[
G(0,\ell)
=
\frac1N
\sum_{j=1}^{N-1}
\frac{\phi_j(-\ell)}{\lambda_j}.
\]
Hence
\[
G(0,0)-G(0,\ell)
=
\frac1N
\sum_{j=1}^{N-1}
\frac{1-\phi_j(-\ell)}{\lambda_j}.
\]
Since
\[
\phi_j(-\ell)
=
\exp\left(-\frac{2\pi i j\ell}{N}\right),
\]
and the left-hand side is real, taking the real part gives
\[
G(0,0)-G(0,\ell)
=
\frac1N
\sum_{j=1}^{N-1}
\frac{1-\cos\left(\frac{2\pi j\ell}{N}\right)}{\lambda_j}.
\]
Thus
\[
h(0,\ell)
=
2k
\sum_{j=1}^{N-1}
\frac{1-\cos\left(\frac{2\pi j\ell}{N}\right)}{\lambda_j}.
\]

Finally, substituting
\[
\lambda_j
=
2k\left\{
1-\frac1k\sum_{r=1}^k
\cos\left(\frac{2\pi jr}{N}\right)
\right\}
\]
into the above formula, we obtain
\[
h(0,\ell)
=
\sum_{j=1}^{N-1}
\frac{1-\cos\left(\frac{2\pi j\ell}{N}\right)}
{1-\frac1k\sum_{r=1}^k\cos\left(\frac{2\pi jr}{N}\right)}.
\]
\end{proof}



\section{Proof of the Main Theorem}
\label{sec05}

In this section, we derive an explicit formula for the average hitting times by using the spectral representation obtained in Section \ref{sec04}.

\begin{lemma} \label{lem51}
Let
\[
\Phi_k(x)=(2-x)\Psi_k(x).
\]
Note that the roots of $\Psi_k(x)$ are distinct \cite{Szego}.
We denote them by
\[
\gamma_1,\dots,\gamma_{k-1}.
\]
Then
\[
\frac{2k}{\Phi_k(x)}
=
\frac{B}{2-x}
+
\sum_{a=1}^{k-1}
\frac{A_a}{\gamma_a-x}
\]
for some constants $B,A_1,\dots,A_{k-1}$.
These coefficients are given by
\[
B=\frac{2k}{\Psi_k(2)},
\qquad
A_a
=
-\frac{2k}{(2-\gamma_a)\Psi_k'(\gamma_a)}
\qquad (a=1,\dots,k-1).
\]
\end{lemma}

\begin{proof}
By Lemma \ref{lem32}, we have
\[
\Phi_k(x)=(2-x)\Psi_k(x).
\]
Since the roots of $\Psi_k(x)$ are assumed to be distinct, we can write
\[
\Psi_k(x)
=
c\prod_{a=1}^{k-1}(x-\gamma_a)
\]
with some constant $c$.
Thus
\[
\Phi_k(x)
=
(2-x)c\prod_{a=1}^{k-1}(x-\gamma_a).
\]

Hence the poles of the rational function
\[
\frac{2k}{\Phi_k(x)}
\]
are only
\[
x=2,\quad x=\gamma_1,\dots,\gamma_{k-1}.
\]
Moreover, all these poles are simple.
Therefore, by the standard partial fraction decomposition for rational functions with only simple poles, there exist constants $B,A_1,\dots,A_{k-1}$ such that
\[
\frac{2k}{\Phi_k(x)}
=
\frac{B}{2-x}
+
\sum_{a=1}^{k-1}
\frac{A_a}{\gamma_a-x}.
\]

First we determine $B$.
Multiplying both sides of the partial fraction decomposition 
by $2-x$, we obtain
\[
(2-x)\frac{2k}{\Phi_k(x)}
=
B
+
(2-x)\sum_{a=1}^{k-1}
\frac{A_a}{\gamma_a-x}.
\]
Taking the limit $x\to2$, 
we obtain 
\[
B=\frac{2k}{\Psi_k(2)}.
\]

Next we determine $A_a$.
Multiplying both sides of the partial fraction decomposition by $\gamma_a-x$, we obtain
\[
(\gamma_a-x)\frac{2k}{\Phi_k(x)}
=
(\gamma_a-x)\frac{B}{2-x}
+
A_a
+
(\gamma_a-x)
\sum_{\substack{b=1\\ b\neq a}}^{k-1}
\frac{A_b}{\gamma_b-x}.
\]
Taking the limit $x\to\gamma_a$, we obtain 
\[
A_a
=
-\frac{2k}{(2-\gamma_a)\Psi_k'(\gamma_a)}.
\]
\end{proof}

By Lemma \ref{lem51}, the rational function
\[
\frac{2k}{\Phi_k(x)}
\]
which appears in the spectral representation is decomposed into the term corresponding to $x=2$ and the correction terms corresponding to the roots $\gamma_a$ of $\Psi_k$.
We next evaluate each correction term by using Fourier analysis and periodization.

\begin{lemma} \label{lem53}
Let
\[
\gamma=\rho+\rho^{-1},
\qquad |\rho|<1.
\]
Then
\[
\sum_{j=1}^{N-1}
\frac{1-\cos(\ell\theta_j)}
{\gamma-2\cos\theta_j}
=
N
\frac{(1+\rho^N)-(\rho^\ell+\rho^{N-\ell})}
{(\rho^{-1}-\rho)(1-\rho^N)}
\]
holds, where
\[
\theta_j=\frac{2\pi j}{N}.
\]
\end{lemma}

\begin{proof}
Let 
\[
g(m)=\frac{\rho^{|m|}}{\rho^{-1}-\rho}.
\]
This $g$ satisfies 
\[
(\gamma-S)g=\delta_0.
\]
Here
\[
(Sf)(m)=f(m+1)+f(m-1),
\]
and $\delta_0$ is defined by
\[
\delta_0(m)
=
\begin{cases}
1 & (m=0),\\
0 & (m\neq 0).
\end{cases}
\]

Then we periodize $g$ and put
\[
G(m)=\sum_{q\in\mathbb{Z}}g(m+qN).
\]
Since $|\rho|<1$, this sum converges absolutely.

For $0\le m\le N$, we compute $G(m)$ explicitly.
By definition,
\[
G(m)
=
\sum_{q\in\mathbb{Z}}
\frac{\rho^{|m+qN|}}{\rho^{-1}-\rho}
=
\frac{1}{\rho^{-1}-\rho}
\sum_{q\in\mathbb{Z}}\rho^{|m+qN|}.
\]
We divide the sum into the parts $q\ge0$ and $q\le -1$:
\[
\sum_{q\in\mathbb{Z}}\rho^{|m+qN|}
=
\sum_{q=0}^{\infty}\rho^{m+qN}
+
\sum_{q=1}^{\infty}\rho^{qN-m}.
\]
Here we used
\[
|m-qN|=qN-m
\]
for $0\le m\le N$ and $q\ge1$.

Hence
\[
G(m)
=
\frac{1}{\rho^{-1}-\rho}
\left(
\sum_{q=0}^{\infty}\rho^{m+qN}
+
\sum_{q=1}^{\infty}\rho^{qN-m}
\right).
\]
Each sum is a geometric series. Thus
\[
\sum_{q=0}^{\infty}\rho^{m+qN}
=
\rho^m\sum_{q=0}^{\infty}(\rho^N)^q
=
\frac{\rho^m}{1-\rho^N},
\]
and
\[
\sum_{q=1}^{\infty}\rho^{qN-m}
=
\rho^{N-m}\sum_{q=0}^{\infty}(\rho^N)^q
=
\frac{\rho^{N-m}}{1-\rho^N}.
\]
Therefore
\[
G(m)
=
\frac{1}{\rho^{-1}-\rho}
\cdot
\frac{\rho^m+\rho^{N-m}}{1-\rho^N}
=
\frac{\rho^m+\rho^{N-m}}
{(\rho^{-1}-\rho)(1-\rho^N)}.
\]

On the other hand, $G(m)$ admits the Fourier representation
\[
G(m)
=
\frac1N
\sum_{j=0}^{N-1}
\frac{e^{im\theta_j}}
{\gamma-2\cos\theta_j}
\]
by discrete Fourier inversion.
It follows that
\[
N\{G(0)-G(\ell)\}
=
\sum_{j=0}^{N-1}
\frac{1-e^{i\ell\theta_j}}
{\gamma-2\cos\theta_j}.
\]
Since the left-hand side is real, taking the real part gives
\[
N\{G(0)-G(\ell)\}
=
\sum_{j=0}^{N-1}
\frac{1-\cos(\ell\theta_j)}
{\gamma-2\cos\theta_j}.
\]
The term for $j=0$ is equal to $0$, since the numerator is $0$.
Hence
\[
\sum_{j=1}^{N-1}
\frac{1-\cos(\ell\theta_j)}
{\gamma-2\cos\theta_j}
=
N\{G(0)-G(\ell)\}.
\]

Finally, substituting
\[
G(0)
=
\frac{1+\rho^N}
{(\rho^{-1}-\rho)(1-\rho^N)}
\]
and
\[
G(\ell)
=
\frac{\rho^\ell+\rho^{N-\ell}}
{(\rho^{-1}-\rho)(1-\rho^N)}
\]
into the above equality, we obtain
\[
\sum_{j=1}^{N-1}
\frac{1-\cos(\ell\theta_j)}
{\gamma-2\cos\theta_j}
=
N
\frac{(1+\rho^N)-(\rho^\ell+\rho^{N-\ell})}
{(\rho^{-1}-\rho)(1-\rho^N)}.
\]
\end{proof}

\begin{lemma} \label{lem54}
For each $a=1,\dots,k-1$, let $\rho_a$ satisfy
\[
\rho_a+\rho_a^{-1}=\gamma_a,
\qquad |\rho_a|<1.
\]
Define
\[
V_n^{(a)}
:=
\frac{\rho_a^n-\rho_a^{-n}}
{\rho_a-\rho_a^{-1}}.
\]
Then
\[
V_{n+1}^{(a)}
=
\gamma_aV_n^{(a)}
-
V_{n-1}^{(a)}
\]
holds.
\end{lemma}

\begin{proof}
The recurrence relation follows from the definition of 
$V_{n}^{(a)}$ 
and the identity
$\rho_a+\rho_a^{-1}=\gamma_a$ 
by direct calculation.
\end{proof}

\begin{theorem} \label{thm51}
For the simple random walk on the cycle power graph $C_N^k$, the average hitting time is given by
\[
h(0,\ell)
=
\frac{6}{(k+1)(2k+1)}\,\ell(N-\ell)
+
N\sum_{a=1}^{k-1}
A_a
\frac{
V_\ell^{(a)}V_{N-\ell}^{(a)}
}{
V_N^{(a)}
}.
\]
Here
\[
A_a
=
-\frac{2k}{(2-\gamma_a)\Psi_k'(\gamma_a)},
\]
and
\[
V_n^{(a)}
=
\frac{\rho_a^n-\rho_a^{-n}}
{\rho_a-\rho_a^{-1}},
\qquad
\rho_a+\rho_a^{-1}=\gamma_a,
\quad |\rho_a|<1.
\]
\end{theorem}

\begin{proof}
By the spectral representation in Proposition \ref{prop41}, we have
\[
h(0,\ell)
=
2k\sum_{j=1}^{N-1}
\frac{1-\cos(\ell\theta_j)}
{\Phi_k(2\cos\theta_j)}.
\]

By Lemma \ref{lem32},
\[
\Phi_k(x)=(2-x)\Psi_k(x).
\]
By Lemma \ref{lem51}, we have the partial fraction decomposition
\[
\frac{2k}{\Phi_k(x)}
=
\frac{B}{2-x}
+
\sum_{a=1}^{k-1}
\frac{A_a}{\gamma_a-x},
\]
where 
\[
B=\frac{2k}{\Psi_k(2)},
\qquad
A_a
=
-\frac{2k}{(2-\gamma_a)\Psi_k'(\gamma_a)}.
\]

We first compute $B$ explicitly.
By the proof of Lemma \ref{lem32}, we have
\[
\Phi_k(x)
=
2k-2\sum_{r=1}^k T_r(x/2).
\]
Therefore
\[
\Phi_k'(2)
=
-2\sum_{r=1}^k T_r'(1)\cdot\frac12
=
-\sum_{r=1}^k T_r'(1).
\]
Using
\[
T_r'(1)=r^2,
\]
we get
\[
\Phi_k'(2)
=
-\sum_{r=1}^k r^2
=
-\frac{k(k+1)(2k+1)}{6}.
\]
On the other hand, since
\[
\Phi_k(x)=(2-x)\Psi_k(x),
\]
we have
\[
\Phi_k'(2)=-\Psi_k(2).
\]
Thus
\[
\Psi_k(2)
=
\frac{k(k+1)(2k+1)}{6}.
\]
It follows that
\[
B
=
\frac{2k}{\Psi_k(2)}
=
\frac{12}{(k+1)(2k+1)}.
\]

Substituting the partial fraction decomposition into the spectral representation, we obtain
\[
h(0,\ell)
=
B
\sum_{j=1}^{N-1}
\frac{1-\cos(\ell\theta_j)}
{2-2\cos\theta_j}
+
\sum_{a=1}^{k-1}
A_a
\sum_{j=1}^{N-1}
\frac{1-\cos(\ell\theta_j)}
{\gamma_a-2\cos\theta_j}.
\]

For the first term, we use the standard identity
\[
\sum_{j=1}^{N-1}
\frac{1-\cos(\ell\theta_j)}
{2-2\cos\theta_j}
=
\frac12\,\ell(N-\ell).
\]
Thus
\[
B
\sum_{j=1}^{N-1}
\frac{1-\cos(\ell\theta_j)}
{2-2\cos\theta_j}
=
\frac{B}{2}\ell(N-\ell)
=
\frac{6}{(k+1)(2k+1)}\,\ell(N-\ell).
\]

For the second term, by Lemma \ref{lem53}, if
\[
\rho_a+\rho_a^{-1}=\gamma_a,
\qquad |\rho_a|<1,
\]
then
\[
\sum_{j=1}^{N-1}
\frac{1-\cos(\ell\theta_j)}
{\gamma_a-2\cos\theta_j}
=
N
\frac{(1+\rho_a^N)-(\rho_a^\ell+\rho_a^{N-\ell})}
{(\rho_a^{-1}-\rho_a)(1-\rho_a^N)}.
\]
Since
\[
(1+\rho_a^N)-(\rho_a^\ell+\rho_a^{N-\ell})
=
(1-\rho_a^\ell)(1-\rho_a^{N-\ell}),
\]
and by direct calculation,
\[
\frac{
V_\ell^{(a)}V_{N-\ell}^{(a)}
}{
V_N^{(a)}
}
=
\frac{
(1-\rho_a^\ell)(1-\rho_a^{N-\ell})
}{
(\rho_a^{-1}-\rho_a)(1-\rho_a^N)
},
\]
we obtain
\[
\sum_{j=1}^{N-1}
\frac{1-\cos(\ell\theta_j)}
{\gamma_a-2\cos\theta_j}
=
N
\frac{
V_\ell^{(a)}V_{N-\ell}^{(a)}
}{
V_N^{(a)}
}.
\]

Combining these formulas, we get
\[
h(0,\ell)
=
\frac{6}{(k+1)(2k+1)}\,\ell(N-\ell)
+
N\sum_{a=1}^{k-1}
A_a
\frac{
V_\ell^{(a)}V_{N-\ell}^{(a)}
}{
V_N^{(a)}
}.
\]
This completes the proof.
\end{proof}

\section{Applications of the Main Theorem}
\label{sec06}

In this section, by using the formula for the average hitting times obtained in Section \ref{sec05}, we derive formulas for the effective resistances, the numbers of spanning trees, the numbers of two-component spanning forests, and the numbers of spanning trees of vertex-identified graphs.


First, we discuss the relation with effective resistance.

\begin{theorem} \label{thm62}
The effective resistance between the vertices $0$ and $\ell$ in $C_N^k$ is given by
\[
R(0,\ell)
=
\frac{h(0,\ell)}{Nk}.
\]
In particular, by Theorem \ref{thm51}, we obtain
\[
R(0,\ell)
=
\frac{6}{Nk(k+1)(2k+1)}\,\ell(N-\ell)
+
\frac1k
\sum_{a=1}^{k-1}
A_a
\frac{
V_\ell^{(a)}V_{N-\ell}^{(a)}
}{
V_N^{(a)}
}.
\]
\end{theorem}

\begin{proof}
By the commute time identity due to Chandra et al.~\cite{ChandraEtAl1989}, for any connected graph, we have
\[
h(u,v)+h(v,u)=2mR(u,v),
\]
where $m$ is the number of edges.

Since $C_N^k$ is vertex-transitive, we have
\[
h(u,v)=h(v,u).
\]
Thus
\[
2h(u,v)=2mR(u,v),
\]
and hence
\[
h(u,v)=mR(u,v).
\]
Since \(C_N^k\) is \(2k\)-regular,
the number of edges is
\[
m=Nk.
\]
Therefore
\[
R(0,\ell)
=
\frac{h(0,\ell)}{Nk}.
\]
Finally, substituting the formula in Theorem \ref{thm51}, we obtain
\[
R(0,\ell)
=
\frac{6}{Nk(k+1)(2k+1)}\,\ell(N-\ell)
+
\frac1k
\sum_{a=1}^{k-1}
A_a
\frac{
V_\ell^{(a)}V_{N-\ell}^{(a)}
}{
V_N^{(a)}
}.
\]
\end{proof}


Next, we compute the number of spanning trees.
For a graph $G$, We denote by $\tau(G)$ the number of spanning trees of $G$.
\begin{theorem} \label{thm63}
The number of spanning trees of the cycle power graph $C_N^k$ is given by
\[
\tau(C_N^k)
=
\frac1N
\prod_{j=1}^{N-1}
\Phi_k\!\left(
2\cos\frac{2\pi j}{N}
\right).
\]
\end{theorem}

\begin{proof}
By Kirchhoff's Matrix-Tree Theorem, the number of spanning trees of a connected graph $G$ is given by
\[
\tau(G)
=
\frac1{|V(G)|}
\prod_{j=1}^{|V(G)|-1}\lambda_j,
\]
where
\[
\lambda_1,\dots,\lambda_{|V(G)|-1}
\]
are the nonzero Laplacian eigenvalues of $G$ \cite{Kirchhoff1847}.

By Lemma \ref{lem41} and the identity
\[
\Phi_k(2\cos\theta)
=
2k-2\sum_{r=1}^k\cos(r\theta),
\]
we have
\[
\lambda_j
=
\Phi_k\left(2\cos\frac{2\pi j}{N}\right).
\]
Therefore
\[
\tau(C_N^k)
=
\frac1N
\prod_{j=1}^{N-1}
\Phi_k\!\left(
2\cos\frac{2\pi j}{N}
\right).
\]
\end{proof}


Next, by using the factorization obtained in Sections \ref{sec03} and \ref{sec05}, we rewrite the number of spanning trees in terms of the roots of $\Psi_k$.

\begin{lemma} \label{lem64}
We have
\[
\prod_{j=1}^{N-1}
\left(
2-2\cos\frac{2\pi j}{N}
\right)
=
N^2.
\]
\end{lemma}

\begin{proof}
The left-hand side is the product of the nonzero Laplacian eigenvalues of the cycle graph $C_N$.

Since
\[
\tau(C_N)=N,
\]
Kirchhoff's Matrix-Tree Theorem gives
\[
N
=
\frac1N
\prod_{j=1}^{N-1}
\left(
2-2\cos\frac{2\pi j}{N}
\right).
\]
Hence
\[
\prod_{j=1}^{N-1}
\left(
2-2\cos\frac{2\pi j}{N}
\right)
=
N^2.
\]
\end{proof}

\begin{theorem} \label{thm65}
Let
\[
\gamma_1,\dots,\gamma_{k-1}
\]
be the roots of $\Psi_k(x)$, and take $\rho_a$ satisfying
\[
\rho_a+\rho_a^{-1}=\gamma_a,
\qquad
|\rho_a|<1.
\]
Then
\[
\tau(C_N^k)
=
N
\prod_{a=1}^{k-1}
\frac{(1-\rho_a^N)^2}
{\rho_a^{N-1}(1-\rho_a)^2}.
\]
\end{theorem}

\begin{proof}
By Lemma \ref{lem32},
\[
\Phi_k(x)=(2-x)\Psi_k(x).
\]
Thus, by Theorem \ref{thm63},
\[
\tau(C_N^k)
=
\frac1N
\prod_{j=1}^{N-1}
(2-x_j)\Psi_k(x_j),
\qquad
x_j=2\cos\frac{2\pi j}{N}.
\]
By Lemma \ref{lem64},
\[
\prod_{j=1}^{N-1}(2-x_j)=N^2.
\]
Hence
\[
\tau(C_N^k)
=
N
\prod_{j=1}^{N-1}
\Psi_k(x_j).
\]

Since
\[
\Psi_k(x)
=
\prod_{a=1}^{k-1}(x-\gamma_a),
\]
we have
\[
\tau(C_N^k)
=
N
\prod_{a=1}^{k-1}
\prod_{j=1}^{N-1}
(x_j-\gamma_a).
\]

Now put
\[
x_j=z_j+z_j^{-1},
\qquad
z_j=e^{2\pi ij/N}.
\]
Then
\[
x_j-\gamma_a
=
z_j+z_j^{-1}
-
(\rho_a+\rho_a^{-1})
=
\frac{(z_j-\rho_a)(z_j-\rho_a^{-1})}{z_j}.
\]
Therefore
\[
\prod_{j=1}^{N-1}(x_j-\gamma_a)
=
\prod_{j=1}^{N-1}
\frac{(z_j-\rho_a)(z_j-\rho_a^{-1})}{z_j}.
\]

Using
\[
\prod_{j=1}^{N-1}(z_j-\rho_a)
=
\frac{1-\rho_a^N}{1-\rho_a}
\]
and
\[
\prod_{j=1}^{N-1}(z_j-\rho_a^{-1})
=
\rho_a^{1-N}
\frac{1-\rho_a^N}{1-\rho_a},
\]
we obtain
\[
\prod_{j=1}^{N-1}(x_j-\gamma_a)
=
\frac{(1-\rho_a^N)^2}
{\rho_a^{N-1}(1-\rho_a)^2}.
\]
Consequently,
\[
\tau(C_N^k)
=
N
\prod_{a=1}^{k-1}
\frac{(1-\rho_a^N)^2}
{\rho_a^{N-1}(1-\rho_a)^2}.
\]
\end{proof}


Finally, we discuss the relation with two-component spanning forests.

\begin{theorem} \label{thm66}
Let $F_{C_N^k}(0 \mid \ell)$ denote the number of two-component spanning forests of $C_N^k$ such that the vertices $0$ and $\ell$ belong to different connected components.
Then
\[
F_{C_N^k}(0 \mid \ell)
=
\tau(C_N^k)\frac{h(0,\ell)}{Nk}.
\]
In particular,
\[
F_{C_N^k}(0 \mid \ell)
=
\frac{\tau(C_N^k)}{Nk}
\left\{
\frac{6}{(k+1)(2k+1)}\,\ell(N-\ell)
+
N\sum_{a=1}^{k-1}
A_a
\frac{
V_\ell^{(a)}V_{N-\ell}^{(a)}
}{
V_N^{(a)}
}
\right\}.
\]
\end{theorem}

\begin{proof}
By the matrix-forest theorem, for any connected graph, we have
\[
F_G(u \mid v)=\tau(G)R(u,v)
\]
\cite{ChebotarevShamis1997}.
By Theorem \ref{thm62},
\[
R(0,\ell)=\frac{h(0,\ell)}{Nk}.
\]
Hence
\[
F_{C_N^k}(0 \mid \ell)
=
\tau(C_N^k)\frac{h(0,\ell)}{Nk}.
\]
Finally, substituting the formula in Theorem \ref{thm51}, we obtain
\[
F_{C_N^k}(0 \mid \ell)
=
\frac{\tau(C_N^k)}{Nk}
\left\{
\frac{6}{(k+1)(2k+1)}\,\ell(N-\ell)
+
N\sum_{a=1}^{k-1}
A_a
\frac{
V_\ell^{(a)}V_{N-\ell}^{(a)}
}{
V_N^{(a)}
}
\right\}.
\]
\end{proof}


\begin{corollary} \label{cor67}
Let $C_N^k/(0\sim \ell)$ be the graph obtained from $C_N^k$ by identifying the vertices $0$ and $\ell$.
Then
\[
\tau(C_N^k/(0\sim \ell))
=
\tau(C_N^k)\frac{h(0,\ell)}{Nk}.
\]
In particular,
\[
\tau(C_N^k/(0\sim \ell))
=
\frac{\tau(C_N^k)}{Nk}
\left\{
\frac{6}{(k+1)(2k+1)}\,\ell(N-\ell)
+
N\sum_{a=1}^{k-1}
A_a
\frac{
V_\ell^{(a)}V_{N-\ell}^{(a)}
}{
V_N^{(a)}
}
\right\}.
\]
\end{corollary}

\begin{proof}
By Chaiken's all-minors Matrix-Tree Theorem, we have
\[
\tau(G/(u\sim v))
=
F_G(u\mid v)
\]
\cite{Chaiken1982}.
Thus, by Theorem \ref{thm66},
\[
\tau(C_N^k/(0\sim \ell))
=
F_{C_N^k}(0\mid \ell)
=
\tau(C_N^k)\frac{h(0,\ell)}{Nk}.
\]
The last formula follows by substituting the explicit expression in Theorem \ref{thm66}.
\end{proof}

\subsection{Examples} \label{sec07}

In this section, we describe the cases $k=1,2,3$ explicitly by using the main theorem.
In particular, Fibonacci numbers appear in the case $k=2$, and complex conjugate second-order linear recurrence sequences appear in the case $k=3$.

\subsection{The case $k=1$}

For $k=1$, we have $C_N^1=C_N$.
In this case,
\[
\Psi_1(x)=1.
\]
Hence there is no root, and no correction term appears in the main theorem.

Therefore
\[
h(0,\ell)
=
\frac{6}{(1+1)(2\cdot1+1)}\,\ell(N-\ell)
=
\ell(N-\ell).
\]

Thus we recover the classical formula
\[
h(0,\ell)=\ell(N-\ell).
\]

\subsection{The case $k=2$}

For $k=2$, we have
\[
P_1(x)=x,\qquad P_2(x)=x^2-2.
\]
Thus
\begin{align*}
\Phi_2(x)
&=4-\{x+(x^2-2)\} \\
&=(2-x)(x+3).
\end{align*}
Therefore
$
\Psi_2(x)=x+3,
$
and its root is
$
\gamma=-3.
$
In this case, the solution of
\[
\rho+\rho^{-1}=-3
\]
satisfying $|\rho|<1$ is
\[
\rho=\frac{-3+\sqrt5}{2}.
\]
Moreover, put
\[
V_n
=
\frac{\rho^n-\rho^{-n}}
{\rho-\rho^{-1}}.
\]
Then, by Lemma \ref{lem54}, it satisfies
\[
V_{n+1}
=
-3V_n-V_{n-1}.
\]
By Theorem \ref{thm51}, we obtain
\[
h(0,\ell)
=
\frac{2}{5}\,\ell(N-\ell)
-
\frac{4}{5}\,N
\frac{V_{\ell}V_{N-\ell}}{V_N}.
\]
Since
\[
\rho=-\phi^{-2},
\qquad
\phi=\frac{1+\sqrt5}{2},
\]
Binet's formula gives
\[
V_n=(-1)^{n+1}F_{2n}.
\]
Therefore
\[
\frac{V_\ell V_{N-\ell}}{V_N}
=
-\frac{F_{2\ell}F_{2(N-\ell)}}{F_{2N}}.
\]
It follows that
\[
h(0,\ell)
=
\frac{2}{5}\,\ell(N-\ell)
+
\frac{4}{5}\,N
\frac{F_{2\ell}F_{2(N-\ell)}}{F_{2N}}.
\]

This agrees with the formula for the average hitting times on the square of a cycle graph obtained by Doi et al.~\cite{DoiEtAl2022}.

\subsection{The case $k=3$} \label{subsec07}

In this subsection, we explicitly describe the average hitting times, effective resistances, numbers of spanning trees, numbers of two-component spanning forests, and numbers of spanning trees of vertex-identified graphs for the third power graph $C_N^3$ of a cycle.
Throughout this subsection, we assume that $N\ge 7$.

For $k=3$, we have
\[
P_1(x)=x,\qquad
P_2(x)=x^2-2,\qquad
P_3(x)=x^3-3x.
\]
Hence
\begin{align*}
\Phi_3(x)
&=
6-\{x+(x^2-2)+(x^3-3x)\} \\
&=
(2-x)(x^2+3x+4),
\end{align*}
and therefore
\[
\Psi_3(x)=x^2+3x+4,
\]
and
its roots are
\[
\gamma_\pm
=
\frac{-3\pm i\sqrt7}{2}.
\]

For each $\gamma_\pm$, take $\rho_\pm$ satisfying
\[
\rho_\pm+\rho_\pm^{-1}=\gamma_\pm,
\qquad
|\rho_\pm|<1.
\]
Furthermore, define
\[
V_n^{(\pm)}
=
\frac{\rho_\pm^n-\rho_\pm^{-n}}
{\rho_\pm-\rho_\pm^{-1}}.
\]
Then, by Lemma \ref{lem54},
\[
V_{n+1}^{(\pm)}
=
\gamma_\pm V_n^{(\pm)}-V_{n-1}^{(\pm)}
\]
holds.
Moreover,
\[
\Psi_3'(x)=2x+3,
\]
and hence
\[
\Psi_3'(\gamma_+)=i\sqrt7,
\qquad
\Psi_3'(\gamma_-)=-i\sqrt7.
\]
Therefore
\[
A_\pm
=
-\frac{6}{(2-\gamma_\pm)\Psi_3'(\gamma_\pm)}
\]
gives
\[
A_+
=
-\frac{3}{14}(1-i\sqrt7),
\qquad
A_-
=
-\frac{3}{14}(1+i\sqrt7).
\]

Then we have the following: 
\begin{corollary}\label{thm71}
The average hitting time from the vertex $0$ to the vertex $\ell$ on $C_N^3$ is given by
\[
h(0,\ell)
=
\frac{3}{14}\ell(N-\ell)
+
2N\operatorname{Re}
\left(
A_+
\frac{
V_\ell^{(+)}V_{N-\ell}^{(+)}
}{
V_N^{(+)}
}
\right).
\]
\end{corollary}

\begin{proof}
Substituting $k=3$ into Theorem \ref{thm51}, we obtain
\[
h(0,\ell)
=
\frac{3}{14}\ell(N-\ell)
+
N
\left(
A_+
\frac{
V_\ell^{(+)}V_{N-\ell}^{(+)}
}{
V_N^{(+)}
}
+
A_-
\frac{
V_\ell^{(-)}V_{N-\ell}^{(-)}
}{
V_N^{(-)}
}
\right).
\]

Since
\[
A_-
\frac{
V_\ell^{(-)}V_{N-\ell}^{(-)}
}{
V_N^{(-)}
}
=
\overline{
A_+
\frac{
V_\ell^{(+)}V_{N-\ell}^{(+)}
}{
V_N^{(+)}
}
},
\]
we have 
\[
A_+
\frac{
V_\ell^{(+)}V_{N-\ell}^{(+)}
}{
V_N^{(+)}
}
+
A_-
\frac{
V_\ell^{(-)}V_{N-\ell}^{(-)}
}{
V_N^{(-)}
}
=
2\operatorname{Re}
\left(
A_+
\frac{
V_\ell^{(+)}V_{N-\ell}^{(+)}
}{
V_N^{(+)}
}
\right).
\]

This completes the proof.
\end{proof}

\begin{corollary} \label{cor72}
The effective resistance between the vertices $0$ and $\ell$ in $C_N^3$ is given by
\[
R(0,\ell)
=
\frac{1}{14N}\ell(N-\ell)
+
\frac23
\operatorname{Re}
\left(
A_+
\frac{
V_\ell^{(+)}V_{N-\ell}^{(+)}
}{
V_N^{(+)}
}
\right).
\]
\end{corollary}

\begin{proof}
By Theorem \ref{thm62},
\[
R(0,\ell)=\frac{h(0,\ell)}{3N}.
\]

Substituting Corollary \ref{thm71} into this identity, we obtain
\[
R(0,\ell)
=
\frac{1}{14N}\ell(N-\ell)
+
\frac23
\operatorname{Re}
\left(
A_+
\frac{
V_\ell^{(+)}V_{N-\ell}^{(+)}
}{
V_N^{(+)}
}
\right).
\]
\end{proof}

\begin{corollary} \label{thm73}
The number of spanning trees of $C_N^3$ is given by
\[
\tau(C_N^3)
=
N
\left|
\frac{
(1-\rho_+^N)^2
}{
\rho_+^{N-1}(1-\rho_+)^2
}
\right|^2.
\]
\end{corollary}

\begin{proof}
Applying Theorem \ref{thm65} with $k=3$, we obtain
\[
\tau(C_N^3)
=
N
\prod_{\sigma=\pm}
\frac{
(1-\rho_\sigma^N)^2
}{
\rho_\sigma^{N-1}(1-\rho_\sigma)^2
}.
\]

Since
\[
\rho_-=\overline{\rho_+},
\]
the two factors are complex conjugates of each other.
Therefore
\[
\tau(C_N^3)
=
N
\left|
\frac{
(1-\rho_+^N)^2
}{
\rho_+^{N-1}(1-\rho_+)^2
}
\right|^2.
\]
\end{proof}

\begin{remark}
The formula for $\tau(C_N^3)$ also appears in \cite{Miezaki}. 
\end{remark}

The spanning tree formula in Corollary \ref{thm73}
is consistent with known product formulas for circulant graphs
(see, for example, \cite{BozkurtEtAl}).
However, the present derivation arises naturally from the explicit formulas for hitting times and the associated Fourier-analytic structure, and is unified with the corresponding formulas for effective resistances, two-component spanning forests, and spanning trees of vertex-identified graphs.

\begin{corollary} \label{thm74}
Let $F_{C_N^3}(0 \mid \ell)$ denote the number of two-component spanning forests of $C_N^3$ such that the vertices $0$ and $\ell$ belong to different connected components.

Then
\[
F_{C_N^3}(0 \mid \ell)
=
\tau(C_N^3)
\left\{
\frac{1}{14N}\ell(N-\ell)
+
\frac23
\operatorname{Re}
\left(
A_+
\frac{
V_\ell^{(+)}V_{N-\ell}^{(+)}
}{
V_N^{(+)}
}
\right)
\right\}.
\]
\end{corollary}

\begin{proof}
By Theorem \ref{thm66},
\[
F_{C_N^3}(0 \mid \ell)
=
\tau(C_N^3)R(0,\ell).
\]

Substituting Corollary \ref{cor72}, we obtain
\[
F_{C_N^3}(0 \mid \ell)
=
\tau(C_N^3)
\left\{
\frac{1}{14N}\ell(N-\ell)
+
\frac23
\operatorname{Re}
\left(
A_+
\frac{
V_\ell^{(+)}V_{N-\ell}^{(+)}
}{
V_N^{(+)}
}
\right)
\right\}.
\]
\end{proof}

\begin{corollary} \label{cor75}
Let $C_N^3/(0\sim \ell)$ be the graph obtained from $C_N^3$ by identifying the vertices $0$ and $\ell$.

Then
\[
\tau(C_N^3/(0\sim \ell))
=
\tau(C_N^3)
\left\{
\frac{1}{14N}\ell(N-\ell)
+
\frac23
\operatorname{Re}
\left(
A_+
\frac{
V_\ell^{(+)}V_{N-\ell}^{(+)}
}{
V_N^{(+)}
}
\right)
\right\}.
\]
\end{corollary}

\begin{proof}
By Corollary \ref{cor67},
\[
\tau(C_N^3/(0\sim \ell))
=
F_{C_N^3}(0 \mid \ell).
\]

Therefore the assertion follows from Corollary \ref{thm74}.
\end{proof}

\begin{remark}
Why do second-order linear recurrence relations naturally appear for cycle power graphs? 
In forthcoming papers \cite{{Miezaki-Tamura-2},{Miezaki-Tamura-3}}, we provide a natural interpretation of this phenomenon using the kernel of the Laplacian operator and the theory of association schemes. 
Furthermore, in the same papers, we investigate this problem together with directed and weighted analogues of the present work.
\end{remark}

\section*{Acknowledgments}
The authors are supported by JSPS KAKENHI (25K06927).


\end{document}